\documentclass{amsart}

\usepackage{ stmaryrd }

\usepackage{amsmath}
\usepackage{amssymb}
\usepackage{amsthm}
\usepackage{amsfonts,epsfig}
\usepackage{mathptmx}

\usepackage{latexsym}
\usepackage{mathrsfs}
\usepackage[utf8]{inputenc}
\usepackage[T1]{fontenc}
\usepackage{ae,aecompl}
\usepackage{braket}
\usepackage{color}

\usepackage{tikz}
\usepackage{tikz-cd}
\usepackage{tikz-3dplot}
\usetikzlibrary{shapes.geometric}
\usetikzlibrary{intersections}
\usetikzlibrary{calc,angles,positioning,intersections,quotes,decorations.markings}

\usepackage{mathrsfs}

\usepackage[
backend=biber,
style=alphabetic,
sorting=ynt
]{biblatex}
\usepackage{hyperref}
\usepackage{subfiles}

\theoremstyle{plain}
\newtheorem{definition}{Definition}[section]
\newtheorem{theorem}[definition]{Theorem}
\newtheorem{lemma}[definition]{Lemma}
\newtheorem{corollary}[definition]{Corollary}
\newtheorem{proposition}[definition]{Proposition}

\theoremstyle{definition}

\newtheorem{remark}[definition]{Remark}

\newcommand{\R}{\mathbb{R}}
\newcommand{\C}{\mathbb{C}}

\newcommand{\HH}{\mathbb{H}}
\newcommand{\g}{\mathfrak{g}}

\DeclareMathOperator{\PSL}{PSL}
\DeclareMathOperator{\PSU}{PSU}
\DeclareMathOperator{\Id}{Id}
\DeclareMathOperator{\Isom}{Isom}

\DeclareMathOperator{\Ep}{Ep}
\newcommand{\comment}[1]{}

\title[AW extension, AC curve and Epstein--Poincar\'{e} Surfaces]{Ahlfors--Weill Extension, Asymptotically Conformal Curves and Epstein--Poincar\'{e} Surfaces}
\author{Ming Hong Tee}
\address{Ming Hong Tee, Department of Mathematics, Boston College,
Chestnut Hill, Ma 02467}
\date{}

\begin{document}

\maketitle
\begin{abstract}
    The Ahlfors--Weill construction explicitly extends a univalent function $f$ with Schwarzian norm less than $1/2$ to a quasiconformal homeomorphism of the Riemann sphere, an extension that can be understood geometrically through Epstein surfaces. For asymptotically conformal curves, the uniformization map can be extended to a larger disk, and Loewner theory can be used to verify that this extension is quasiconformal. In this paper, we construct a local quasireflection with such an extension. Using this local quasireflection, we modify classical estimates of Ahlfors and show that domains bounded by asymptotically conformal curves have asymptotically similar hyperbolic and quasihyperbolic metrics. Furthermore, by combining these estimates with a geometric reinterpretation of the extension through Epstein surfaces, we show that the Epstein--Poincaré surfaces of the complementary domains are arbitrarily close to one another near the curve.
\end{abstract}

\section{Introduction}
The classical theorem of Nehari \cite{nehari} establishes conditions on the norm of the Schwarzian derivative of a locally univalent function $f$ defined on the unit disk under which $f$ is univalent. Subsequently, Ahlfors and Weil \cite{ahlforsweill} showed, through an explicit construction, that if the Schwarzian norm of $f$ is less than $1/2$, then $f$ admits a quasiconformal extension to the entire Riemann sphere. Later, Pommerenke \cite{Pommerenke1978} considered the class of asymptotically conformal curves---a subclass of quasiconformal curves---and showed that the corresponding uniformizing map $f$ from the unit disk onto the Jordan domain bounded by such a curve admits a quasiconformal extension to a larger disk. 

In \cite{ahlforsQuasiRef}, Ahlfors studied Jordan curves $\gamma$ in the Riemann sphere that admit a quasireflection, namely a quasiconformal map $j$ that permutes the two complementary Jordan domains and fixes $\gamma$ pointwise. For such a reflection, he proved three estimates relating a point $w$, its image $j(w)$, and the curves $\gamma$, with a constant $C(K)$ that depends only on the quasiconformal constant $K$ of $j$. 

In this paper, we study Ahlfors's bounds in the context of asymptotically conformal curves. To do so, we use the local quasiconformal extension results of Becker and Pommerenke \cite{PomBecker1978} and of Becker \cite{becker1987asymptotically}. Using these extensions, we construct a local quasireflection $j$, and we show that the constants in Ahlfors's estimates converge to $1$ as $w$ converges to $\gamma$. As a consequence, we show that the hyperbolic and quasihyperbolic metrics are asymptotically similar for domains bounded by asymptotically conformal curves.

In \cite{epstein-envelopes}, Epstein constructed natural surfaces in the hyperbolic three-space $\HH^3$ associated with a conformal metric on an open set in the Riemann sphere. We call such a surface the Epstein-Poincar\'{e} surface if the open set is equipped with the hyperbolic metric. If the open set is the image of a univalent map $f$ with Schwarzian norm less than $1/2$, using the curvature of the Epstein-Poincar\'{e} surface, Epstein showed that the normal flow of the surface foliates $\HH^3$ and limits to the Ahlfors--Weill construction, thus showing that the Ahlfors--Weill extension is a homeomorphism. 

In this paper, we study Epstein-Poincar\'{e} surfaces in the context of asymptotically conformal curves. As in the case where Schwarzian norm less than $1/2$, these surfaces tie the properties of the domain at the conformal boundary to the geometry of the surfaces in $\HH^3$. We illustrate this by giving a geometric proof of Becker's theorem (Theorem \ref{thm: local homeomorphism of AW extension}). Furthermore, using the modified bounds of Ahlfors and the ideas from this geometric proof, we show that the Epstein-Poincar\'{e} surfaces of the complementary domains are arbitrarily close to one another near the curve. 

Finally, we return to the case where the Schwarzian norm of $f$ is less than $1/2$ and compare the bounds found throughout the paper. More precisely, we establish a global bound on the quantity $EN(w)$ of Earle and Nag \cite{Earle:qcreflection} in terms of the Schwarzian norm of $f$, without assuming that the curve is asymptotically conformal. 

\section{Results} \label{sec: result}
Let $\Omega$ be a bounded Jordan domain in $\C$ and let $f$ a biholomorphic function from the unit disk $\Delta$ onto $\Omega$. The Schwarzian derivative of $f$ is defined as 
\begin{align*}
    Sf(z) =\left(\left(\frac{f''}{f'}\right)'(z) - \frac{1}{2}\left(\frac{f''}{f'}\right)^2(z)\right)dz^2.
\end{align*}
Note that $Sf$ is a holomorphic quadratic differential on $\Delta$. For any holomorphic quadratic differential $q$ defined on $\Delta$, the norm of $q$ is given by
\begin{align*}
    \|q\|_\Delta = \sup_{z\in\Delta} \|q(z)\|_\Delta
\end{align*}
where $\|q(z)\|_\Delta = \frac{|q(z)|}{\rho_\Delta(z)}$ and $\rho_\Delta(z) = \frac{4}{(1-|z|^2)^2}|dz|^2$ is the hyperbolic metric on $\Delta$. The Schwarzian norm of a locally univalent function $f$ is given by $\|Sf\|_\Delta$. 

The following theorem due to Nehari provides a necessary and a sufficient condition for a locally univalent function $f:\Delta \to \C$ to be univalent. 
\begin{theorem}(See \cite{nehari})
    Let $f:\Delta \to \C$ be a locally univalent function. If $f$ is univalent, its Schwarzian norm is less than $\frac{3}{2}$. Conversely, if $f:\Delta\to \C$ is a locally univalent function with $\|Sf\|_\Delta<\frac{1}{2}$, then $f$ is univalent.
\end{theorem} 

The Ahlfors--Weill extends Nehari's theorem as follows:
\begin{theorem}(See \cite{ahlforsweill}, \cite{ahlforsQuasiRef}) \label{thm: AW thm}
    Given any holomorphic quadratic differential $q$ on $\Delta$ with $\|q\|_\Delta <\frac{1}{2}$, there exists a univalent function $f:\Delta \to \C$ such that $Sf = q$ and $f$ admits a quasiconformal extension to the Riemann sphere.
\end{theorem}

The construction is given as follows: For each $z\in\Delta$, the osculating M\"{o}bius transformation of $f$ at $z$ is the unique fractional linear transformation $Mf(z,\cdot)$ that satisfies $Mf(z,z) = f(z)$, $Mf'(z,z) = f'(z)$ and $Mf''(z,z) = f''(z)$ where the prime denotes $\partial/\partial z$. The Ahlfors--Weill extension of $f$ is given by 
\begin{align*}
    F(z) = \begin{cases}
        f(z) &, z\in \overline{\Delta}, \\
        Mf(\frac{1}{\bar{z}},z) &,z\in \Delta^*,
    \end{cases}
\end{align*}
where $\Delta^*$ is the complement of the closed unit disk $\bar{\Delta}$ in the Riemann sphere $\hat{\C}$.
By a direct computation (for instance \cite{ahlforsQuasiRef}), the Beltrami differential of $F$ satisfies
\begin{align*}
    \mu_F(z) = \frac{F_{\bar{z}}}{F_z}(z)
    = \begin{cases}
        0 &, z\in \Delta, \\
        -2\frac{Sf(\frac{1}{\bar{z}})}{\rho_\Delta(\frac{1}{\bar{z}})} &, z\in \Delta^*. 
    \end{cases}
\end{align*}

In this paper, we investigate when $F$ is a quasiconformal homeomorphism from a neighborhood of the closed disk $\overline{\Delta}$ onto its image. The natural class of Jordan domains consists of the domains bounded by asymptotically conformal curves.

Let $\gamma$ be a closed Jordan curve in $\C$, and let $C(w_1,w_2)$ be the arc of $\gamma$ between $w_1$ and $w_2$ with the smaller Euclidean diameter. The curve $\gamma$ is quasiconformal if 
\begin{align*}
    \max_{w\in C(w_1,w_2)} \frac{|w_2-w| + |w-w_1|}{|w_2-w_1|}<\infty 
\end{align*}
as $|w_2-w_1| \to 0$; this holds if and only if $\gamma$ is the image of the unit circle under a quasiconformal map of $\C$ (see Theorem 9.14 of \cite{PommerenkeUniFn}). On the other hand, in \cite{Pommerenke1978}, $\gamma$ is an asymptotically conformal curve if and only if it satisfies
\begin{align*}
    \max_{w\in C(w_1,w_2)} \frac{|w_2-w| + |w-w_1|}{|w_2-w_1|}\to 1 
\end{align*}
as $|w_2-w_1| \to 0$. The following proposition shows that asymptotically conformal curves are candidates for the study of Ahlfors--Weill construction.

\begin{proposition}(See \cite{Pommerenke1978}, \cite{PomBecker1978})\label{prop: asymptotically conformal curve}
    Let $\gamma$ be a Jordan curve in $\C$ and $\Omega$ be the bounded domain enclosed by $\gamma$. Let $f:\Delta \to \Omega$ be a biholomorphic map. The following are equivalent:
    \begin{enumerate}
        \item The curve $\gamma$ is asymptotically conformal, 
        \item $\lim_{|z|\to 1^-}\left(\frac{f''}{f'}\right)(z)(1-|z|^2) = 0,$ 
        \item $\lim_{|z|\to 1^-} \|Sf(z)\|_\Delta = 0,$
        \item there exists a quasiconformal extension $F$ of $f$ to a larger disk such that $\mu_F(z) \to 0$ as $|z| \to 1^+$.
    \end{enumerate}
\end{proposition}
The quantity $\frac{f''(z)}{f'(z)}$ is known as the nonlinearity of $f$ and we will denote it by $Nf$. The equivalence of (1) and (2) is shown in Theorem 1 of \cite{Pommerenke1978} and the equivalence of (2), (3) and (4) is shown in Theorem 1 of \cite{PomBecker1978}. Note that in the proof of Theorem 1 of \cite{PomBecker1978}, Becker and Pommerenke show the implication from (2) to (4) by constructing a desired quasiconformal extension through the use of the Loewner differential equation. The extension they consider is 
\begin{align*}
    P(z) = \begin{cases}
        f(z) &, z\in \Delta, \\
        f(\frac{1}{\bar{z}}) + (z-\frac{1}{\bar{z}})f'(\frac{1}{\bar{z}}) &,z\in \Delta^*.
    \end{cases}
\end{align*}

Denote $\Delta_r = r\Delta$ and $A_{a,b} = \Delta_b - \overline{\Delta_a}$. In line with our discussion of Ahlfors--Weill, Becker \cite{becker1987asymptotically} shows that the Ahlfors--Weill construction does in fact give a quasiconformal extension.

\begin{theorem}(See \cite{becker1987asymptotically}) \label{thm: local homeomorphism of AW extension}
    For any asymptotically conformal curve $\gamma$ and $f:\Delta \to \Omega$ a biholomorphic map onto a bounded Jordan domain enclosed by $\gamma$, there exists $r<1$ such that $f$ admits a quasiconformal extension $F$ from $\Delta_{1/r}$ onto its image in $\hat{\C}$ satisfying 
    $$ F(z) = \begin{cases}
        f(z) &, z\in \Delta, \\
        Mf(\frac{1}{\bar{z}},z) &,z\in A_{1,1/r}.
    \end{cases}$$
\end{theorem}

In \cite{ahlforsQuasiRef}, Ahlfors studies a Jordan curve in $\hat{\C}$ that admits a quasiconformal reflection, that is, a quasiconformal map that permutes the two Jordan domains and fixes the Jordan curve pointwise. For such a reflection $j$, he has shown the following estimate:
\begin{lemma}(See \cite{ahlforsQuasiRef})
    For all $w_0 \in \gamma$,
    \begin{itemize}
        \item $\frac{1}{C(K)} \leq \left|\frac{w-w_0}{j(w)-w_0}\right| \leq C(K),$
        \item $\frac{|w-j(w)|}{2d(w,\gamma)} \leq C(K),$
        \item $\frac{1}{C(K)}\leq \frac{d(w,\gamma)}{d(j(w),\gamma)} \leq C(K).$
    \end{itemize}
    where $C(K)$ is a constant depending only on $K$, the $L^\infty$ norm of $\frac{j_z}{j_{\bar{z}}}$. 
\end{lemma}
Let $\rho_\Omega$ be the hyperbolic metric on $\Omega$. The quasihyperbolic metric $\sigma_\Omega$ on $\Omega$ is defined as $\frac{|dw|^2}{d(w,\gamma)^2}$. From the Schwarz lemma and the Koebe $1/4$ theorem, we obtain the classical bound
\begin{align*}
    \frac{1}{4}\sigma_\Omega(w) \leq \rho_\Omega(w) \leq 4\sigma_\Omega(w)
\end{align*}
for all $w\in \Omega$.

For an asymptotically conformal curve, by constructing a local quasireflection from the local quasiconformal extension in Theorem \ref{thm: local homeomorphism of AW extension}, we will show that all the constants involved converge to $1$ as $w$ converges to the boundary of $\Omega$. Furthermore, we will show the following estimate:
\begin{theorem}\label{thm: hyp and quasi hyp}
    Let $\gamma$ be an asymptotically conformal curve in $\C$ and let $\Omega$ be the bounded Jordan domain enclosed by $\gamma$. The hyperbolic metric and quasihyperbolic metric on $\Omega$ are asymptotically equal at the boundary of $\Omega$. Formally,
    \begin{align*}
        \lim_{w\to \partial \Omega} \frac{\rho_{\Omega}(w)}{\sigma_\Omega(w)} = 1.
    \end{align*}
\end{theorem}
This comparison will be further discussed in Section \ref{sec: Hyp and QuasiHyp}.

In \cite{Earle:qcreflection}, Earle and Nag study the Ahlfors--Weill construction for a general domain that admits a quasiconformal reflection $j$. Let $\Omega^* = \hat{\C}-\bar{\Omega}$ with hyperbolic metric $\rho_{\Omega^*}$. We define the quantity $EN(w)$ with respect to $j$ by
\[ EN(w) = |w-j(w)|^4\rho_\Omega(w) \rho_{\Omega^*}(j(w)) \]
for all $w\in \Omega$. In Lemma 3 of \cite{Earle:qcreflection}, they have shown the following estimate:
\begin{lemma} (See \cite{Earle:qcreflection})
    Let $j$ be a $K$-quasiconformal reflection interchanging the complementary Jordan domains $\Omega$ and $\Omega^*$. There is a constant $C(K)$, depending only on $K$, such that 
    \begin{align*}
        C(K)^{-1} \leq EN(w) \leq C(K)
    \end{align*}
    for all $w\in \Omega$.
\end{lemma}
When $\Omega$ is a Euclidean disk in $\C$ and $j$ is the antiholomorphic reflection (i.e $j(w)=1/\bar{w}$ for $\Omega = \Delta$), the function $EN(w) \equiv 16$. 

In this paper, we can define a local quasireflection $j(w) = F\left(\frac{1}{\overline{F^{-1}(w)}}\right)$ and a similar term $EN(w)$ for an asymptotically conformal curve via Theorem \ref{thm: local homeomorphism of AW extension} and show the following theorem:
\begin{theorem} \label{thm: EN converges}
    For any asymptotically conformal curve $\gamma$, the quantity $EN(w)$ with respect to $j$ converges to $16$ as $w$ converges to the boundary of $\Omega$.
\end{theorem}

In \cite{epstein-envelopes}, Epstein constructed immersed surfaces in $\HH^3$ corresponding to conformal metric on a domain at $\hat{\C}$; the construction and the properties of Epstein surface we need are recalled in \ref{sec: Ep}. This construct relates the interior $\HH^3$ with the conformal boundary $\hat{\C}$. An important class of metric would be the Poincar\'{e} metric. Epstein showed that when the Schwarzian norm of the uniformization map is less than $1/2$, the surfaces obtained from the Epstein--Poincar\'{e} surface by the normal flow foliate $\HH^3$. This leads to a geometric proof of Ahlfors--Weill theorem. 

In this paper, we will study Epstein surfaces for asymptotically conformal curve. Similar to the geometric proof of Ahlfors--Weill theorem, in Section \ref{sec: AW}, we will give a geometric proof of Becker's theorem (Theorem \ref{thm: AW thm}) using Epstein surfaces. 

In Section \ref{sec: dist btw Ep}, we will return to the quantity $EN(w)$ with respect to a general quasireflection $j$. We first show that $EN(w)$ is related to the hyperbolic distance between the horoballs based at $w$ and $j(w)$ determined by the hyperbolic metric on $\Omega$ and $\Omega^*$, respectively. For asymptotically conformal curve, using this interpretation and the local quasireflection above, we will establish the following result:
\begin{theorem} \label{thm: Eps converges}
    For any asymptotically conformal Jordan curve $\gamma$,
    \begin{align*}
        \lim_{w\to \partial \Omega}d_{\HH^3}(\Ep_{\Omega}(w),\Ep_{\Omega^*}(\Omega^*)) = 0.
    \end{align*}
\end{theorem}

In Section \ref{sec: global bound}, we return to the case $\|Sf\|_\Delta < 1/2$ without assuming that the Jordan curve is asymptotically conformal. In this setting, by Theorem \ref{thm: AW thm}, the Ahlfors--Weill construction gives a quasiconformal extension of $f$ to the whole Riemann sphere. Therefore, $j(w) = F^{-1}\left(\frac{1}{\overline{F(w)}}\right)$ is a quasireflection on the Riemann sphere that fixes the Jordan curve pointwise. The main result of Section \ref{sec: global bound} is the following:
\begin{theorem} \label{thm: global bound on EN}
    For a Jordan domain $\Omega$ with a uniformization map $f:\Delta \to \Omega$ satisfying $\|Sf\|_\Delta<1/2$, the function EN(w) satisfies
    \begin{align*}
        EN(w) = 4|w-j(w)|^2\rho_{\Omega^*}(j(w))|j_{\bar{z}}(w)| \leq \frac{16}{1-4\|Sf\|_\Delta^2}
    \end{align*}
    for all $w\in \Omega$.
\end{theorem}

\section{Hyperbolic and Quasihyperbolic metrics} \label{sec: Hyp and QuasiHyp}
In this section, we will prove Theorem \ref{thm: hyp and quasi hyp}. Throughout this section, we will assume that $\gamma$ is an asymptotically conformal curve. Let $f:\Delta \to \Omega$ be a biholomorphic map. Recall that the osculating M\"{o}bius transformation of $f$ at $z$ is the unique fractional linear transformation $Mf(z,\cdot)$ satisfies $Mf(z,z) = f(z)$, $Mf'(z,z) = f'(z)$ and $Mf''(z,z) = f''(z)$. By an explicit computation, 
\[ Mf(z,p) = f(z) + \frac{f'(z)(p-z)}{1-Nf(z)\frac{p-z}{2}}\]
where $Nf(z) = \frac{f''(z)}{f'(z)}$. 

By Theorem \ref{thm: local homeomorphism of AW extension}, there exists $r<1$ such that $F:A_{r,1/r} \to \C$ is a homeomorphism onto its image. Let $j(w) = F\left(\frac{1}{\overline{F^{-1}(w)}}\right)$ for all $w\in F(A_{r,1/r})$. Note that $j^2 = \Id$ and $j|_\gamma = \Id$. Throughout the rest of the paper, for any open set $U\subset \C$ with bounded boundary $\partial U$, $z\to \partial U$ represents any sequence $z_n \in U$ with $d(z_n,\partial U) \to 0$ where $d$ the Euclidean distance. The next few lemmas are similar in nature to the estimate in \cite{ahlforsQuasiRef}.

\begin{lemma} \label{lem: metric and diff com for F}
    $\lim_{z\to \partial \Delta}\frac{|F(z)-j(F(z))|(\rho_\Omega(F(z)))^{1/2}}{2} = 1$.
\end{lemma}
\begin{proof}
    By the explicit computation of $F$ and Proposition \ref{prop: asymptotically conformal curve},
    \begin{align*}
        |F(z)-j(F(z))|(\rho_{\Omega}(F(z))^{1/2}/2
        &=\left|\frac{F(z) - F(\frac{1}{\bar{z}})}{F'(z)(1-|z|^2)}\right| \\
        &= \left|\frac{f'(z)(\frac{1}{\bar{z}}-z)}{f'(z)(\frac{1}{\bar{z}}-z)(1-Nf(z)\frac{(\frac{1}{\bar{z}}-z)}{2})}\right|\frac{1}{|z|} \\
        &= \frac{1}{|z|}\frac{1}{|1-Nf(z)\frac{(\frac{1}{\bar{z}}-z)}{2}|} \to 1
    \end{align*}
    as $z \to \partial \Delta$.

    \comment{
    and 
    \begin{align*}
        |(r(G(z_1))-G(z_1))(\rho_{\Omega^*}(G(z_1))^{1/2}|/2
        &=|\frac{G(z_1) - G(\frac{1}{\bar{z_1}})}{G'(z_1)(1-|z_1|^2)}| \\
        &= |\frac{g'(z_1)(\frac{1}{\bar{z_1}}-z_1)}{g'(z_1)(\frac{1}{\bar{z_1}}-z_1)(1-Ng(z_1)\frac{(\frac{1}{\bar{z_1}}-z_1)}{2})}|\frac{1}{|z_1|} \\
        &= \frac{1}{|z_1|}\frac{1}{|1-Ng(z_1)\frac{(\frac{1}{\bar{z_1}}-z_1)}{2}|} \to 1
    \end{align*}
    as $z \to \partial \Delta$ where $z_1 = G^{-1}(j(F(z)))$.
    }    
\end{proof}

Note that for $z\in A_{1,1/r}$,
    \begin{align*}
        \partial_z F(z) &= \partial_z(Mf(\frac{1}{\bar{z}},z)) \\
        &= \frac{f'(\frac{1}{\bar{z}})(1-Nf(\frac{1}{\bar{z}})\frac{z-\frac{1}{\bar{z}}}{2}) - f'(\frac{1}{\bar{z}})(z-\frac{1}{\bar{z}})(\frac{-Nf(\frac{1}{\bar{z}})}{2})}
        {(1-Nf(\frac{1}{\bar{z}})\frac{z-\frac{1}{\bar{z}}}{2})^2} \\
        &= \frac{f'(\frac{1}{\bar{z}})}{(1-Nf(\frac{1}{\bar{z}})\frac{z-\frac{1}{\bar{z}}}{2})^2}.
    \end{align*}

\begin{lemma} \label{lem: comparison for F}
    For all $z\in A_{r,1}$, let $\alpha,\beta\in \gamma$ such that $d(\gamma,j(F(z))) = |j(F(z))-\alpha|$ and $d(\gamma,F(z)) = |F(z)-\beta|$. Then, $$\lim_{z\to\partial \Delta} \frac{d(\gamma,F(z))}{d(\gamma,j(F(z)))} = \limsup_{z\to\partial \Delta} \frac{|F(z)-\alpha|}{d(\gamma,j(F(z)))} = \liminf_{z\to\partial \Delta} \frac{d(\gamma,F(z))}{|j(F(z))-\beta|} = 1.$$
\end{lemma}
\begin{proof}
    Since $F$ is an embedding on $A_{r,1/r}$, by compactness of $\gamma$, there exists $r'>0$ such that the $r'$-neighborhood $N$ of $\gamma$ is contained in $F(A_{r,1/r})$. Let $U = F^{-1}(N)$. Note that for all $z\in U\cap \Delta$, the image of the path $\tau(t) = tj(F(z)) + (1-t)\alpha$ for $0<t<1$ lies in $F(A_{1,1/r})$. Since $j$ is $C^2(A_{r,1/r}-S^1)$ and $j$ is a reflection that fixes $\gamma$ pointwise,
    \begin{align*}
        d(\gamma,F(z))
        &\leq |F(z) -\alpha| \\
        &= |j^2(F(z)) - j(\alpha)| \\
        &\leq \int_0^1 \left|\frac{d}{dt}j(\tau(t))\right|dt \\
        &\leq \int_0^1 \left|j_z(\tau(t)) + j_{\bar{z}}(\tau(t))\frac{\overline{\tau'(t)}}{\tau'(t)}\right| |\tau'(t)|dt \\
        &\leq \|j_{\bar{z}}\|_{\infty,\tau}\left(1 + \|\frac{j_z}{j_{\bar{z}}}\|_{\infty,\tau}\right)|j(F(z)) - \alpha|
    \end{align*}
    where $\| h\|_{\infty,\tau} = \sup_{0\leq t\leq 1} \|h(\tau(t))\|$. For $z\in A_{1,1/r}$,
    \begin{align*}
        |j_{\bar{z}}(F(z))| &= \left|F_z\left(\frac{1}{\overline{z}}\right)\left(\frac{-1}{\overline{z}^2}\right)\overline{F^{-1}_z(F(z))}\right| \\
        &= \left|F_z\left(\frac{1}{\overline{z}}\right)\left(\frac{-1}{\overline{z}^2}\right)\frac{1}{\overline{F_z(z)}(1-|\mu_F(z)|^2)}\right| \\
        &= \left|\left(\frac{f_z(\frac{1}{\bar{z}})}{\overline{f_z(\frac{1}{\bar{z}})}}\right)\left(\frac{-1}{\overline{z}^2}\right)\frac{\left(\overline{1-Nf(\frac{1}{\bar{z}})\frac{(z-\frac{1}{\bar{z}})}{2}}\right)^2}{1-|\mu_F(z)|^2}\right| \to 1
    \end{align*} 
    as $z \to \partial \Delta^*$.
    Likewise,
    \begin{align*}
        \frac{|j_z(F(z))|}{|j_{\bar{z}}(F(z))|} = 
        |\mu_{F^{-1}}(F(z))| = |\mu_F(z)| \to 0
    \end{align*}
    as $z \to \partial\Delta^*$. Hence,
    \begin{align*}
        \limsup_{z\to \partial \Delta} \frac{d(\gamma,F(z))}{d(\gamma,j(F(z)))} \leq \limsup_{z\to \partial \Delta} \frac{|F(z)-\alpha|}{d(\gamma,j(F(z)))} \leq 1.
    \end{align*}

    A similar argument shows that 
    \begin{align*}
        \limsup_{z\to \partial \Delta} \frac{d(\gamma,j(F(z)))}{d(\gamma,F(z))} \leq \limsup_{z\to \partial \Delta} \frac{|j(F(z))-\beta|}{d(\gamma,F(z))} \leq 1.
    \end{align*}
\end{proof}

\begin{lemma} \label{lem: diff and dist com for F}
    $\lim_{z\to \partial \Delta}\frac{|F(z)-j(F(z))|}{2d(\gamma,j(F(z)))} = 1.$
\end{lemma}
\begin{proof}
    For all $z\in A_{r,1}$, let $\alpha$ be a point in $\gamma$ such that $d(\gamma,j(F(z))) = |j(F(z))-\alpha|$. Since $j$ fixes $\gamma$ pointwise, by the triangle inequality,
    \begin{align*}
        \frac{|F(z) - j(F(z))|}{d(\gamma,j(F(z)))} &\leq \frac{|F(z) - \alpha|}{d(\gamma,j(F(z)))} + 1. 
    \end{align*}
    By Lemma \ref{lem: comparison for F}, $\limsup_{z\to \partial \Delta}\frac{|F(z)-j(F(z))|}{2d(\gamma,j(F(z)))} \leq 1$.

    For the other inequality,
    \begin{align*}
        \frac{2d(\gamma,j(F(z)))}{|F(z)-j(F(z))|} = \frac{d(\gamma,j(F(z)))}{\delta(z)}\frac{2\delta(z)}{|F(z)-j(F(z))|} \leq \frac{d(\gamma,j(F(z)))}{\delta(z)}
    \end{align*}
    where $\delta(z) = \min(d(\gamma,j(F(z))),d(\gamma,F(z)))$.
    Since $\frac{d(\gamma,j(F(z))}{\delta(z)} = \max (1,\frac{d(\gamma,j(F(z))}{d(\gamma,F(z))})$, by Lemma \ref{lem: comparison for F},
    \begin{align*}
        \limsup_{z\to \partial \Delta} \frac{2d(\gamma,j(F(z)))}{|F(z)-j(F(z))|} \leq \lim_{z\to \partial \Delta}\frac{d(\gamma,j(F(z)))}{\delta(z)} = \lim_{z\to \partial \Delta}\frac{d(\gamma,j(F(z)))}{d(\gamma,F(z))} = 1. 
    \end{align*}
\end{proof}

\begin{proof}[Proof of Theorem \ref{thm: hyp and quasi hyp}]
    Note that 
    \begin{align*}
        \rho_{\Omega}(w)d(w,\gamma)^2 = \frac{|w-j(w)|^2\rho_\Omega(w)}{4}\frac{4d(j(w),\gamma)^2}{|w-j(w)|^2}\frac{d(w,\gamma)^2}{d(j(w),\gamma)^2}.
    \end{align*}
    The conclusion follows from Lemmas \ref{lem: metric and diff com for F}, \ref{lem: diff and dist com for F} and \ref{lem: comparison for F}.
\end{proof}

\begin{remark}
    One may in fact use the local quasireflection $j_P(w) = P\left(\frac{1}{\overline{P^{-1}(w)}}\right)$ where $P$ is the extension given in the Section \ref{sec: result} and a similar approach can be used to prove all the statements mentioned in this section. However, because the connection between extension $F$ and the Epstein-Poincar\'{e} surface is needed in Section \ref{sec: dist btw Ep}, we will only prove all the statements in this section using $F$. 
    
    In the Appendix, we will derive a metric that is naturally related to $P\left(\frac{1}{\overline{P^{-1}(w)}}\right)$.
\end{remark}

\section{The Earle--Nag Estimate for Asymptotically Conformal Curve}\label{sec: EN est}

For $\gamma$ an asymptotically conformal curve, with the same notation as in the previous section, the function $EN(w)$ on $F(A_{r,1/r})$ with respect to $j$ is given by
\[ EN(w) = |w-j(w)|^4\rho_\Omega(w) \rho_{\Omega^*}(j(w)) \]
for all $w \in \Omega \cap F(A_{r,1/r})$. 

For the rest of this section, we will assume that our Jordan curve is asymptotically conformal. The main theorem of this section is Theorem \ref{thm: EN converges}. We begin by discussing the properties of the biholomorphic map $g:\Delta^*\to \Omega^*$ when $\gamma$ is asymptotically conformal. 

\begin{lemma} \label{lem: asy ppty for g}
    For any asymptotically conformal curve $\gamma$, 
    \begin{enumerate}
        \item $\lim_{|z|\to 1^+}Ng(z)(|z|^2-1) = 0$, 
        \item $\lim_{|z|\to 1^+} \|Sg(z)\|_{\Delta^*}=0$.
    \end{enumerate}
\end{lemma}

\begin{proof}
    Recall that an asymptotically conformal curve satisfies
    \begin{align*}
    \max_{w\in C(w_1,w_2)} \frac{|w_2-w| + |w-w_1|}{|w_2-w_1|}\to 1 
\end{align*}
as $|w_2-w_1| \to 0$. Therefore, asymptotic conformality is invariant under translation and complex multiplication. Furthermore, if $\gamma$ does not contain $0$, let $\widetilde{\gamma}$ be the image of $\gamma$ under $1/z$. Then, for $z\in C(z_1,z_2) \subset \widetilde{\gamma}$,
\begin{align*}
    \frac{|z_2-z| + |z-z_1|}{|z_2-z_1|} 
    = \frac{\frac{|w_2-w|}{|w_2w|} + \frac{|w-w_1|}{|w_1w|}}{\frac{|w_2-w_1|}{|w_2w_1|}} 
    = \frac{\frac{|w_1|}{|w|}|w_2-w|+\frac{|w_2|}{|w|}|w_1-w|}{|w_2-w_1|}
\end{align*}
where $w=\frac{1}{z}$, $w_i = \frac{1}{z_i}$, $i=1,2$. Note that, 
\begin{align*}
    \frac{\min(|w_1|,|w_2|)}{|w|} \frac{|w_2-w|+|w_1-w|}{|w_2-w_1|} 
    &\leq \frac{\frac{|w_1|}{|w|}|w_2-w|+\frac{|w_2|}{|w|}|w_1-w|}{|w_2-w_1|} \\
    &\leq \frac{\max(|w_1|,|w_2|)}{|w|} \frac{|w_2-w|+|w_1-w|}{|w_2-w_1|}.
\end{align*}
By the continuity of $1/z$ and the asymptotic conformality of $\gamma$, $\widetilde{\gamma}$ is asymptotically conformal.

In order to show (1) and (2), consider $\psi = \iota \circ t\circ  g\circ \iota$ where $\iota(z) = 1/z$ and $t$ is a translation such that the image of $\gamma$ avoids $0$. Note that the image of $\gamma$ under $\iota\circ t$ is asymptotically conformal and $\psi$ is a biholomorphic map from $\Delta$ to the bounded Jordan domain $\iota(t(\Omega^*))$. By Proposition \ref{prop: asymptotically conformal curve}, 
\begin{enumerate}
    \item $\lim_{|z|\to 1^-}N\psi(z)(1-|z|^2) = 0,$
    \item $\lim_{|z|\to 1^-} \|S\psi(z)\|_\Delta = 0.$
\end{enumerate}
For (1), by direct computation, $$N\psi(z)(1-|z|^2) = -Ng(\frac{1}{z})(1-|z|^2)\frac{1}{z^2} + 2\frac{g'(\frac{1}{z})(1-|z|^2)}{t(g(\frac{1}{z}))z^2} - \frac{2(1-|z|^2)}{z}.$$
Note that $\lim_{|z|\to 1^-} \frac{2(1-|z|^2)}{z} = 0$ and
\begin{align*}
    \lim_{|z|\to 1^-} \frac{|g'(\frac{1}{z})|(1-|z|^2)}{|t(g(\frac{1}{z}))z^2|} 
    = \lim_{|z|\to 1^+} \frac{|g'(z)|(|z|^2-1)}{|t(g(z))|} 
    = \lim_{|z|\to 1^+} \frac{2}{|t(g(z))|\sqrt{\rho_{\Omega^*}(g(z))}} = 0.
\end{align*}
Therefore,
\begin{align*}
    0 = \lim_{|z|\to 1^-} N\psi(z)(1-|z|^2) = \lim_{|z|\to 1^-}Ng(\frac{1}{z})\left(\frac{1}{|z|^2}-1\right) = \lim_{|z|\to 1^+} Ng(z)(|z|^2-1). 
\end{align*}
For (2), it follows from considering $\psi$ and noting that the Schwarzian norm is invariant under pre- and post-composition by M\"{o}bius transformations.
\end{proof}

With Lemma \ref{lem: asy ppty for g} and a similar argument, by increasing $r<1$ if necessary, $g$ extends to a quasiconformal homeomorphism $G$ from $\hat{\C} -\Delta_r$ onto its image in $\hat{\C}$ satisfying $$ G(z) = \begin{cases}
        Mg(\frac{1}{\bar{z}},z) &, z\in A_{r,1}, \\
        g(z) &,z\in \Delta^*.
    \end{cases}$$
    
Likewise, its Beltrami differential satisfies
\begin{align*}
    |\mu_G(z)| = \begin{cases}
        2\|Sg(\frac{1}{\bar{z}})\|_{\Delta^*} &, z\in A_{r,1} \\   
        0 &, z\in \Delta^*. \end{cases}
\end{align*}
Let $j^*(w) = G\left(\frac{1}{\overline{G^{-1}(w)}}\right)$ for all $w\in G(A_{r,1/r})$.

Using similar arguments, we obtain the following lemmas for the quasiconformal map $G$,
\begin{lemma} \label{lem: metric and diff com for G}
    $\lim_{z_1\to \partial\Delta^*} \frac{|(j^*(G(z_1))-G(z_1))(\rho_{\Omega^*}(G(z_1))^{1/2}|}{2} = 1$.
\end{lemma}
\begin{lemma}
    For all $z_1\in \Delta^*$, let $\alpha,\beta\in \gamma$ such that $d(\gamma,G(z_1))) = |G(z_1))-\alpha|$ and $d(\gamma,j^*(G(z_1))) = |j^*(G(z_1))-\beta|$. Then, $\lim_{z_1\to\partial \Delta^*} \frac{d(\gamma,j^*(G(z_1)))}{d(\gamma,G(z_1))} = \limsup_{z_1\to\partial \Delta^*} \frac{|j^*(G(z_1))-\alpha|}{d(\gamma,G(z_1))} = \liminf_{z_1\to\partial \Delta^*} \frac{d(\gamma,j^*(G(z_1)))}{|G(z_1)-\beta|} = 1$.
\end{lemma}

\begin{lemma} \label{lem: diff and dist com for G}
    $\lim_{z_1\to \partial \Delta^*} \frac{2d(\gamma,G(z_1))}{|G(z_1) - j^*(G(z_1))|} = 1$. 
\end{lemma}

\begin{proposition}
    The hyperbolic metric and quasihyperbolic metric on $\Omega^*$ are asymptotically equal at the boundary of $\Omega^*$. Formally,
    \begin{align*}
        \lim_{w\to \partial \Omega^*} \rho_{\Omega^*}(w)d(w,\gamma)^2 = 1.
    \end{align*}
\end{proposition}

\begin{proof}[Proof of Theorem \ref{thm: EN converges}]
    Note that
    \begin{align*}
        &EN(F(z))/16  \\
        &= |F(z)-j(F(z))|^4\rho_\Omega(F(z)) \rho_{\Omega^*}(j(F(z)))/16 \\
        &= \frac{|F(z) - j(F(z))|^2\rho_\Omega(F(z))}{4}
        \frac{|j^*(j(F(z)))-j(F(z))|^2\rho_{\Omega^*}(j(F(z)))}{4}\frac{|F(z) - j(F(z))|^2}{|j^*(j(F(z)))-j(F(z))|^2}. 
    \end{align*}
    By Lemma \ref{lem: metric and diff com for F}, the first term converges to $1$ as $z\to \partial \Delta$. 
    
    For the second term, let $z_1 = G^{-1}(j(F(z)))$. Then, by Lemma \ref{lem: metric and diff com for G},
    \begin{align*}
        \frac{|j^*(j(F(z))-j(F(z))|\rho_{\Omega^*}(j(F(z)))^{1/2}}{2} = \frac{|j^*(G(z_1))-G(z_1)|\rho_{\Omega^*}(G(z_1))^{1/2}}{2}
    \end{align*}
    converges to $1$ as $z\to \partial \Delta$ since $G^{-1}\circ j \circ F$ is continuous.

    For the third term,
    \begin{align*}
        \frac{|F(z) - j(F(z))|}{|j^*(j(F(z))-j(F(z))|} 
        = \frac{|F(z) - j(F(z))|}{2d(\gamma,j(F(z)))}\frac{2d(\gamma,G(z_1))}{|j^*(G(z_1)) - G(z_1)|}
    \end{align*}
    which converges to $1$ as $z\to \partial \Delta$ by Lemmas \ref{lem: diff and dist com for F} and \ref{lem: diff and dist com for G}.

    As $F$ is continuous, $EN(w)/16 \to 1$ as $w\to \partial \Omega$.
\end{proof}

\section{Epstein Surfaces in $\HH^3$} \label{sec: Ep}
In this paper, we give an explicit description of Epstein surfaces in $\HH^3$ and state the required properties for the remaining sections. For a general discussion on Epstein surfaces, see \cite{epstein-envelopes}. 

Let $\mathsf{U}\HH^3$ be the unit tangent bundle of $\HH^3$. For $x\in \HH^3$, the visual metric $\nu_x$ is the pushforward of the round metric on $\mathsf{U}_x\HH^3$ to $\hat{\C}$ via the hyperbolic Gauss map. Geometrically, this means the metric measures the boundary from the viewpoint of an observer inside the hyperbolic space. As the basepoint $x$ moves closer to a region on the boundary, that region takes up a larger field of view, resulting in an increase in the conformal factor near that region.

Alternatively, the visual metric $\nu_x$ on $S^2 = \partial \HH^3$ is defined as follows:
for all $z\in S^2$, consider the geodesic $g$ that passes through $x$ and has endpoint $z$. Consider the hyperbolic plane $P$ in $\HH^3$ that contains $x$ and is perpendicular to $g$. The plane $P$ limits to a circle in $S^2$. Equip the disk $D$, bounded by the circle and containing $z$, with the hyperbolic metric $\rho_D$. Then, $\nu_x(z) = \rho_D(z).$

Let $\rho:\Omega\to \R_{+}$ be a smooth function on a domain $\Omega\subset \C$. Note that for each $w\in \Omega$, the set $H_\rho(w) = \{x\in \HH^3|\nu_x(w) = \rho(w)\}$ is a horosphere. Furthermore, in the upper half-space model $\{(x,y,\xi)\in \R^3 \mid \xi > 0 \}$, $H_\rho(w)$ is in fact a Euclidean sphere of radius $\frac{1}{\sqrt{\rho(w)}}$ in $\HH^3$ tangent to the $xy-$plane at $w$. 

The Epstein surface with respect to $\rho$ is the surface $\Sigma_\rho$ in $\HH^3$ that is the envelope for the horospheres $\{H_\rho(w)\}_{w\in \Omega}$. The following proposition summarizes the properties of a general Epstein surface.
\begin{proposition}\label{prop: properties of general Ep}
    (See \cite{epstein-envelopes})
    There exists a unique smooth map $\widetilde{\Ep}_\rho$ from $\Omega$ into $\mathsf{U}\HH^3$, such that
    \begin{enumerate}
        \item $\widetilde{\Ep}_\rho(w)$ is an inward-pointing normal to $H_\rho(w)$,
        \item if $\pi:\mathsf{U}\HH^3 \to \HH^3$ is the projection $v\mapsto p$ for $v\in \mathsf{U}_p\HH^3$, then $\Ep_\rho = \pi\circ \widetilde{\Ep}_\rho$ is a map from $\Omega$ onto $\Sigma_\rho$,
        \item $\Ep_\rho(w)$ only depends on $\rho$ and its derivative at $w$.
    \end{enumerate}
    Moreover, if we let $\rho_t = e^{2t}\rho$,
    \begin{itemize}
        \item for $\g_t$ the geodesic flow on $\mathsf{U}\HH^3$, $\widetilde{\Ep}_{\rho_t}(w) = \g_{t}(\widetilde{\Ep}_\rho(w))$,
        \item for $\g_{\pm\infty}(v) = \lim_{t\to \pm\infty}\g_t(v) \in S^2$, $\g_{\infty}(\widetilde{\Ep}_\rho(w)) = w.$ 
    \end{itemize}
\end{proposition}

Note that $\PSL(2,\C)$ acts on $\HH^3$ by isometries. The following property relates the action of $\PSL(2,\C)$ and Epstein surfaces.
\begin{proposition}\label{prop: naturality of Ep}
    (See \cite{epstein-envelopes}) For all $w\in \Omega$, $M\in \PSL(2,\C)$, $$\Ep_\rho(w) = M(\Ep_{M^*\rho}(M^{-1}(w))).$$
\end{proposition}

Let $w\in \Omega$ such that there exists an open set $U$ containing $w$ on which $\Ep_\rho|_U$ is an embedding. Let $(x_1,x_2)$ be local coordinates for $U$ and $X_i = d\Ep_\rho(\partial x_i)$.
The first and second fundamental forms of $\Ep_\rho(U)$ are $$g_{ij} = \langle X_i,X_j\rangle, \qquad \Pi_j^i = g^{ik}\Pi_{kj}$$ respectively, where $\Pi_{kj} = \langle \nabla_{X_k}X_j,\widetilde{\Ep}_\rho\rangle$. The principal curvatures $k_1,k_2$ at $\Ep_\rho(w)$ are the eigenvalues of $\Pi_j^i$ at $\Ep_\rho(w)$. Since principal curvatures are invariant under change of basis, by choosing an orthonormal basis, the second fundamental form is symmetric and hence the principal curvatures are real.

For the case where $\Omega$ is a simply connected domain in $\hat{\C}$, we denote by $\rho_\Omega$ the hyperbolic metric on $\Omega$ and call the corresponding Epstein surface the Epstein-Poincar\'{e} surface. We denote $\Ep_\Omega = \Ep_{\rho_\Omega}$. Let $f:\Delta \to \Omega$ be any biholomorphic map. For such a surface, we have the following properties:
\begin{proposition}
    (See \cite{epstein-envelopes}) For each $z\in \Delta$ with $\|Sf(z)\|_\Delta<1$, $\Ep_{\Omega}$ is an embedding near $f(z)$. The principal curvatures at $\Ep_\Omega(f(z))$ (with respect to the inward-pointing normal to the horoball) are given by 
    \begin{align*}
        k_\pm = \frac{\|Sf(z)\|_\Delta}{\|Sf(z)\|_\Delta\pm 1}.
    \end{align*}
\end{proposition}

For the remainder of this section, we will prove some properties of the curves on the Epstein-Poincar\'{e} surface.
\begin{lemma} \label{lem: bound on length and curv}
    Let $\gamma$ be an asymptotically conformal curve. For all $\frac{11-\sqrt{89}}{8}>K>0$, for all $p\in S^1$, there exists an open set $U_p \in \overline\Delta$ such that for all $x\neq y\in U_p-S^1$, there exists a smooth path $\tau$ from an open interval $I$ containing $0$ and $l$ to $\Ep_{\Omega}(\Omega)$ satisfying $\tau(0) = \Ep_{\Omega}(f(x))$, $\tau(l) = \Ep_{\Omega}(f(y))$, $\|\bar{\nabla}_{\tau'}\tau'\| \leq \frac{5K}{2}$ and $\|\tau'\| \geq 1-K $ where $\bar{\nabla}$ is the Levi-Civita connection on the Epstein surface induced from the hyperbolic metric on $\HH^3$.
\end{lemma}

\begin{proof}
    Since $\gamma$ is asymptotically conformal, there exists an open ball $B_r(p)$ of radius $r$ centered at $p$, such that for all $z\in B_r(p)\cap \Delta$, the Schwarzian norm $\|Sf(z)\|_{\Delta}<K$. Let $0<r'<r$ such that the hyperbolic distance from $\partial B_r(p) \cap \Delta$ to $B_{r'}(p)\cap \Delta$ is greater than $1/2$. Let $U_p = B_{r'}(p)\cap \overline{\Delta}$.
    
    For $x\neq y\in U_p$, let $\phi:I\to U_p$ be a geodesic arc segment parameterized by unit speed in $\Delta$ with $\phi(0) = x$, $\phi(l) = y$ where $l = d_{\HH^2}(x,y)$. Let $\tau = \Ep_{\Omega}(f(\phi))$. By the proof of Lemma 3.4 of \cite{KB_Labourie}, $\tau$ satisfies $\|\bar{\nabla}_{\tau'}\tau'\| \leq \frac{5K}{2}$ and $\|\tau'\| \geq 1-K$. 
\end{proof}

We consider the upper half space model for $\HH^3$ with hyperbolic metric $\rho_{\HH^3}(x,y,\xi) = \frac{dx^2+dy^2+d\xi^2}{\xi^2}$. Let $\tau(t) = (x(t),y(t),\xi(t))$.

\begin{lemma} \label{lem: decreasing Eps}
    If $\xi'(0) = 0$, then for all $0<s<l$, $\xi'(s)<0$. In particular, $\xi(s)<\xi(0)$ for $0<s<l$.
\end{lemma}
\begin{proof}
    Consider the diffeomorphism $\varphi:\R^3 \to \HH^3$ given by $(x,y,h) \mapsto (x,y,e^h)$. The pullback of the hyperbolic metric $\rho_{\HH^3}$ via $\varphi$ is given by $\frac{dx^2+dy^2}{e^{2h}} + dh^2$. Let $\lambda(s) = \varphi^{-1}\circ\tau(s) = (x(s),y(s),h(s))$ for all $0<s<l$.

    Let $\partial x = \frac{\partial}{\partial x}$, $\partial y = \frac{\partial}{\partial y}$ and $\partial h = \frac{\partial}{\partial h}$ be the corresponding vector fields in $\R^3$. Recall that the Koszul formula is given by
    \begin{align*}
        2\langle \nabla_XY,Z\rangle = X\langle Y,Z\rangle + Y\langle X,Z\rangle - Z\langle X,Y\rangle + \langle [X,Y],Z\rangle - \langle [Y,Z],X\rangle + \langle [Z,X],Y\rangle
    \end{align*}
    for all $X,Y,Z$ vector fields on $\R^3$. Since the Lie bracket of any two elements in $\{\partial x, \partial y, \partial h\}$ is the zero vector field, we obtain the following computation,    
    \begin{align*}
        \langle X,Y\rangle  &= \begin{cases}
            0 , \text{ if } X\neq Y \in \{\partial x, \partial y,\partial h\},\\
            1, \text{ if } X=Y = \partial h, \\
            e^{-2h}, \text{ if } X=Y \in \{\partial x,\partial y\}.
        \end{cases} \\
        \langle \nabla_X Y,\partial h \rangle &= \begin{cases}
            e^{-2h} , \text{ if } X = Y \in \{\partial x,\partial y\}, \\
            0 , \text{ otherwise,}
        \end{cases}
    \end{align*}
    where $\nabla$ is the Levi-Civita connection on $\HH^3$. Therefore, 
    \begin{align*}
        \langle \nabla_{\lambda'} \lambda',\partial h\rangle = 
        h'' + (x'^2 + y'^2)e^{-2h}.
    \end{align*}
    Note that $\|\lambda'\|^2 = h'^2 + (x'^2 + y'^2)e^{-2h}$ which implies $\langle \nabla_{\lambda'} \lambda',\partial h\rangle = h'' - h'^2 + \|\lambda'\|^2$. On the other hand, if $\bar{\nabla}$ is the Levi-Civita connection on the Epstein surface induced by the hyperbolic metric on $\HH^3$, by splitting the connection into the tangential and normal direction,
    \begin{align*}
        \langle \nabla_{\lambda'}\lambda',\partial h\rangle &= \langle \bar \nabla_{\lambda'}\lambda',\partial h\rangle + \langle \nabla_{\lambda'}\lambda',d\varphi^{-1}(\widetilde{\Ep}_\Omega)\rangle\langle d\varphi^{-1}(\widetilde{\Ep}_\Omega),\partial h\rangle \\
        &\leq  \|\bar \nabla_{\lambda'}\lambda'\| + k_{\max}\|\lambda'\|^2.
    \end{align*}
    where $k_{\max} = \max(|\frac{K}{K\pm 1}|) = \frac{K}{1-K}$ is the maximum of the absolute value of the principal curvatures. Since $k_{max}-1 = \frac{2K-1}{1-K}<0$ and $\varphi$ is an isometry, by Lemma \ref{lem: bound on length and curv}, we have
    \begin{align*}
        h'' -h'^2 \leq \|\bar \nabla_{\lambda'}\lambda'\| + (k_{\max}-1)\|\lambda'\|^2 \leq \frac{5K}{2} + (2K-1)(1-K) < 0.
    \end{align*} 
    The result follows from the next lemma.
    \begin{lemma}
        Suppose $h:(-\epsilon,l+\epsilon) \to \R$ with $h'(0) = 0$ and $h''-h'^2<0$ for all $0< s< l$. Then $h'(s)<0$ for all $0< s< l$.
    \end{lemma}
    
    \begin{proof}
        Consider the function $H(s) = e^{-h(s)}$. Then, $H''(s) = H(s)(-(h''(s)-h'(s)^2)) > 0$. Hence, $H$ is concave up. Since $H'(0) = -h'(0)e^{-h(0)} = 0$, $H'(s) >0$ for all $s\in [0,l]$. Thus, $h'(s) = (-\log(H(s)))' = \frac{-H'(s)}{H(s)}<0$ for all $0< s< l$.  
    \end{proof}
    Since $\xi(s) = e^{h(s)}$, $\xi'(s) = h'(s)e^{h(s)}<0$ for all $0< s< l$. 
\end{proof}

\begin{remark} \label{rmk: bounded by horoballs}
    As a consequence of Lemma \ref{lem: decreasing Eps}, if $\xi'(0) = 0$, for all $s>0$, $\tau(s)$ is below the horoball based at $\infty$ with height $\xi(0)$. By post-composing $\tau$ with M\"{o}bius transformation, we see that for all $s>0$, $\tau(s)$ lies in the same component in the complement of all horoballs that are tangent to the vector $\tau'(0)$. Since the union of all those horoballs contains the hyperbolic plane in $\HH^3$ that contains $\tau(0)$ and is perpendicular to $\tau'(0)$, for all $s>0$, $\tau(s)$ lies in a component bounded by the hyperbolic plane. In particular, $\lim_{s\to \infty}\tau(s)$ is in one of the disks on $\partial \HH^3$ bounded by the hyperbolic plane.
\end{remark}

Recall that we denote the time $t$ geodesic flow starting at $v$ in $\mathsf{U}\HH^3$ by $\g_t(v)$ and $\g_{\pm\infty}(v) = \lim_{t\to \pm\infty}\g_t(v).$
\begin{corollary} \label{cor: flow at infinity are inj}
    For $x\neq y\in U_p-S^1$, $\g_{-\infty}(\widetilde{\Ep_{\Omega}}(f(x))) \neq \g_{-\infty}(\widetilde{\Ep_{\Omega}}(f(y)))$.
\end{corollary}
\begin{proof}
    By post-composing with an isometry, we can assume $\widetilde{\Ep_{\Omega}}(f(x)) = (0,0,-1) \in T_{(0,0,1)}\HH^3$. Therefore, $\tau'(0)$ is perpendicular to $(0,0,1)$ which implies $\xi'(0) = 0$. By Lemma \ref{lem: decreasing Eps}, $\xi'(l) <0$ which implies that $\widetilde{\Ep_{\Omega}}(f(y))$ is not vertical. Thus, $\g_{-\infty}(\widetilde{\Ep_{\Omega}}(f(x))) \neq \g_{-\infty}(\widetilde{\Ep_{\Omega}}(f(y)))$.
\end{proof}

\section{Ahlfors--Weill Extension} \label{sec: AW}
In this section, we will use Epstein surfaces to give a proof of Theorem \ref{thm: local homeomorphism of AW extension}. Let $f:\Delta\to \C$ be an injective holomorphic map. 
Recall that the map $F$ we are interested in is given by 
\begin{align*}
    F(z) = \begin{cases}
        f(z) &, z\in \overline{\Delta}, \\
        Mf(\frac{1}{\bar{z}},z) &,z\in \Delta^*
    \end{cases}
\end{align*}
where $Mf(w,z)$ is the osculating M\"{o}bius transformation of $f$ based at $w$ evaluated at $z$. The following lemma establishes the naturality of the normal-flowed Epstein map.
\begin{lemma} \label{lem: naturality of osc}
    Let $\Omega$ be a bounded simply connected domain in $\C$ and $f:\Delta \to \Omega$ a biholomorphic map. Then, for all $z\in \Delta$, $$\pi(\g_t(\widetilde{\Ep}_\Omega(f(z)))) = Mf(z,\pi(\g_t(\widetilde{\Ep}_{\Delta}(z)))).$$ In particular, $\g_{-\infty}(\widetilde{\Ep}_\Omega(f(z))) = Mf(z,\frac{1}{\bar{z}}) = F(\frac{1}{\bar{z}})$.
\end{lemma}
\begin{proof}
    By Proposition \ref{prop: properties of general Ep} $\widetilde{\Ep}_\Omega(f(z))$ depends only on the $f(z),f'(z)$ and $f''(z)$. Therefore, $\widetilde{\Ep}_\Omega(f(z)) = \widetilde{\Ep}_{Mf(z,\Omega)}(Mf(z,z))$. By Proposition \ref{prop: properties of general Ep}, $$ 
    \pi(\g_t(\widetilde{\Ep}_\Omega(f(z))) = \pi(\g_t(\widetilde{\Ep}_{Mf(z,\Omega)}(Mf(z,z)))) = \pi(\widetilde{\Ep}_{e^{2t}\rho_{Mf(z,\Omega)}}(Mf(z,z))) = \Ep_{e^{2t}\rho_{Mf(z,\Omega)}}(Mf(z,z)).$$
    By Proposition \ref{prop: naturality of Ep},
    \[ \Ep_{e^{2t}\rho_{Mf(z,\Omega)}}(Mf(z,z)) = Mf(z,(\Ep_{e^{2t}\rho_{\Delta}}(z))) = Mf(z,(\pi(\g_t(\widetilde{\Ep}_{\Delta}(z))))).\]
\end{proof}

\begin{lemma} \label{lem: local homeo of F}
    For an asymptotically conformal curve, the map $F$ is a local homeomorphism at each point of $S^1$ in $\C$.
\end{lemma}
\begin{proof}
    For each $p \in S^1$, let $U_p$ be the open set in $\overline{\Delta}$ given in Lemma \ref{lem: bound on length and curv} and let 
    $$V_p =\{z\in \C|z\in U_p \text{ or } \frac{1}{\bar{z}}\in U_p\}$$ be the open set (open disk) in $\C$. Since $f:\overline{\Delta}\to \overline{\Omega}$ is a homeomorphism, $F|_{V_p\cap \overline{\Delta}}$ is injective onto its image.

    We will now show that $F|_{V_p\cap \Delta^*}$ is injective. For all $z\neq w\in V_p\cap \Delta^*$,$ \frac{1}{\overline{z}} \neq \frac{1}{\overline{w}}\in U_p-{S^1}$. By Corollary \ref{cor: flow at infinity are inj} and Lemma \ref{lem: naturality of osc},
    \begin{align*}
        F(z) = \g_{-\infty}(\widetilde{\Ep_{\Omega}}(f(\frac{1}{\bar{z}}))) \neq \g_{-\infty}(\widetilde{\Ep_{\Omega}}(f(\frac{1}{\bar{w}}))) = F(w).
    \end{align*}
    To show that $F|_{V_p}$ is injective, for all $z\in U_p$ and $w\in V_p-U_p$, 
    \begin{enumerate}
        \item If $z\in \Delta$ and $z=\frac{1}{\bar{w}}$, then 
    \begin{align*}
        F(z) = f(z) = \g_{\infty}(\widetilde{\Ep_{\Omega}}(f(z))) \neq \g_{-\infty}(\widetilde{\Ep_{\Omega}}(f(\frac{1}{\bar{w}}))) = F(w).
    \end{align*}
        \item If $z\in \Delta$ and $z\neq \frac{1}{\bar{w}}$, by post-composing with a M\"{o}bius transformation, we can assume $\widetilde{\Ep_{\Omega}}(f(z)) = (0,0,1) \in T_{(0,0,1)}\HH^3$. By an argument similar to that in Corollary \ref{cor: flow at infinity are inj}, 
    \begin{align*}
        \infty = f(z) = \g_{\infty}(\widetilde{\Ep_{\Omega}}(f(z))) \neq \g_{-\infty}(\widetilde{\Ep_{\Omega}}(f(\frac{1}{\bar{w}}))).
    \end{align*}
    Therefore, by Lemma \ref{lem: naturality of osc}, $f(z) \neq F(w)$.
        \item If $z\in S^1$, by post-composing with a M\"{o}bius transformation, we can assume $\widetilde{\Ep_{\Omega}}(f(\frac{1}{\bar{w}})) = (0,0,-1) \in T_{(0,0,1)}\HH^3$. By Lemma \ref{lem: naturality of osc}, $F(w) = \g_{-\infty}(\widetilde{\Ep_{\Omega}}(f(\frac{1}{\bar{w}}))) = \infty$. Consider the unit speed parameterized geodesic $\eta$ from $\frac{1}{\bar{w}}$ to $z$ and choose $z'$ distinct from $\frac{1}{\bar{w}}$ and $z$ on $\eta$. Let $\tau = \Ep_\Omega(f(\eta))$. By Corollary \ref{cor: flow at infinity are inj}, we see that
        $$\g_{-\infty}(\widetilde{\Ep_{\Omega}}(f(\frac{1}{\bar{w}}))) \neq \g_{-\infty}(\widetilde{\Ep_{\Omega}}(f(z'))).$$
        By Remark \ref{rmk: bounded by horoballs}, since the derivative $\tau'$ based at $\Ep_\Omega(f(z'))$ is not horizontal, the limit of $\tau$ which is $z$ lies in the disk bounded by the hyperbolic plane perpendicular to the derivative. Hence, $z\neq \infty = F(w)$.
    \end{enumerate} 

    Note that $F$ is a continuous map. By the invariance of domain theorem, $F|_{V_p}$ is a homeomorphism onto its image.
\end{proof}
\begin{lemma} \label{lem: homeo on annulus}
    Let $h:A_{r,\frac{1}{r}} \to \C$ be a continuous, local homeomorphism at each point of $S^1$ in $A_{r,\frac{1}{r}}$ and assume $h$ is a homeomorphism from $S^1$ onto its image. Then, there exists $r'<1$ such that $h$ is a homeomorphism from $A_{r',1/r'}$ onto its image.
\end{lemma}
\begin{proof}
    Suppose not, for all sequence $r<r_n<1$ with $r_n \to 1$, there exists $x_n\neq y_n \in A_{r_n,1/r_n}$ such that $h(x_n)=h(y_n)$. By compactness, we can assume $x_n\to x$ and $y_n \to y$ in $S^1$. By injectivity and continuity of $h$ on $S^1$, $x = y$. However, local homeomorphism at $x$ shows that for large $n$, $x_n = y_n$, a contradiction.
\end{proof}
\begin{proof}[Proof of Theorem \ref{thm: local homeomorphism of AW extension}]
    Let $0<r'<1$ be given in Lemma \ref{lem: homeo on annulus}. By the pasting lemma, $F$ is a continuous map from $\frac{1}{r'}\Delta$ into $\C$. By Lemma \ref{lem: homeo on annulus} and the univalence of $f$ in $\Delta$, $F$ is locally univalent on $\frac{1}{r'}\Delta$. It suffices to show that $F$ is injective. 
    
    Since $F(\Delta)$ is connected and $F(\Delta)\cap \gamma=\emptyset$, it must be in some connected component of $\C-\gamma$. So too is $F(A_{1,\frac{1}{r'}})$. Hence, it suffices to show that they lie in different connected components.

    Suppose that they lie in the same component. By the Jordan-Schoenflies theorem, there exists a homeomorphism $\psi$ from $\R^2$ to itself that sends $\gamma$ to the unit circle, the bounded component to the open unit disk and the unbounded component to the exterior of the closed unit disk. Since the composition of locally injective maps is locally injective, one can see that $\psi\circ F$ is not locally injective on $S^1$ if $\psi(F(\Delta))$ and $\psi(F(\frac{1}{r'}\Delta))$ are in the same connected component. 

    Recall that the Beltrami differential $\mu_F$ satisfies 
    \begin{align*}
        |\mu_F(z)| = \begin{cases}
            0 &, z\in \Delta,\\
            2\|Sf(\frac{1}{\bar{z}})\|_\Delta &, z\in \Delta^*.
        \end{cases}
    \end{align*}
    By Proposition \ref{prop: asymptotically conformal curve}, the extension $F$ is quasiconformal on $\frac{1}{r'} \Delta$.
\end{proof}

\section{The Distances between Epstein Surfaces}\label{sec: dist btw Ep}
In this section, we will show that the distances between Epstein surfaces are arbitrary close to one another when the Jordan curve is asymptotically conformal, which is formally stated in Theorem \ref{thm: Eps converges}.

To do so, we will first show that for a Jordan curve that admits a quasireflection $j$, the function $EN(w)$ with respect to $j$ is related to the distance between certain horoballs as stated formally below in Proposition \ref{prop: dist btw horoball}. 

\begin{proposition} \label{prop: general dist btw horoball}
    Let $f:\Delta\to \Omega$ and $g:\Delta^*\to \Omega^*$ be uniformization maps. For all $p \in \Delta$ and $q\neq g^{-1}(\infty) \in \Delta^*-\{\infty\}$, the two horoballs $H_{\rho_\Omega}(f(p))$ and $H_{\rho_{\Omega^*}}(g(q))$
    based at $f(p)$ and $g(q)$ respectively have disjoint interiors and the shortest distance between the horoballs, denoted by $D(f,g,p,q)$, in $\HH^3$ is equal to 
    \begin{align*}
        D(f,g,p,q) = \log\left|\frac{(f(p)-g(q))^2}{f'(p)g'(q)(1-|p|^2)(|q|^2 - 1)}\right| \geq 0.
    \end{align*}
\end{proposition}

Before giving the proof of this proposition, we need a special case of the Grunsky's inequality. 
\begin{lemma} \label{lem: Grunsky}(See Theorem 4.1 of \cite{PommerenkeUniFn} or Lemma 4.1 of \cite{BBVW_universal} )
    for $f:\Delta\to \C$ and $g:\Delta^*\to \hat{\C}$ univalent functions with $f(0) = 0$ and $g(\infty) = \infty$, if $f(\Delta) \cap g(\Delta^*) = \emptyset$, then
    \begin{align*}
        \int_\Delta \left| \frac{f'(z)}{f(z)}-\frac{1}{z}\right|^2 dA(z) + \int_{\Delta^*} \left| \frac{g'(z)}{g(z)}-\frac{1}{z}\right|^2 dA(z) \leq 2\pi\log \left| \frac{g'(\infty)}{f'(0)}\right|
    \end{align*}
    with equality if the complement of $f(\Delta)\cup g(\Delta^*)$ has zero Lebesgue measure. 
\end{lemma}
\begin{proof}[Proof of Proposition \ref{prop: general dist btw horoball}]
    We first normalize $p$ and $q$ to $0$ and $\infty$ and keeping track of required the M\"{o}bius transformations. Since $\PSL(2,\C) = \Isom^+(\HH^3)$ and $\PSU(1,1) = \Isom^+(\Delta)=\Isom^+(\Delta^*)$, $D(f,g,p,q) = D(M\circ f \circ m_p,M\circ g \circ m_q, 0,\infty) = \log\left|\frac{(M\circ g \circ m_q)'(\infty)}{(M\circ f \circ m_p)'(0)}\right|$ where
    \begin{align*}
        M(z) = \frac{z-f(p)}{z-g(q)}, 
        m_p(z) = \frac{z+p}{1+\bar{p}z}, 
        m_q(z) = \frac{z + \frac{1}{\bar{q}}}{1+\frac{1}{q}z}.
    \end{align*} 
    By Lemma \ref{lem: Grunsky}, $\log\left|\frac{(M\circ g \circ m_q)'(\infty)}{(M\circ f \circ m_p)'(0)}\right|\geq 0$. 

    The derivative of $M, m_q$ and $m_p$ are given by 
    \begin{align*}
        M'(z) = \frac{f(p)-g(q)}{(z-g(q))^2}, 
        m_p'(z) = \frac{1-p\bar{p}}{(1+\bar{p}z)^2}, 
        m_q'(z) = \frac{1-\frac{1}{q\bar{q}}}{(1+\frac{1}{q}z)^2}.
    \end{align*}    
    Therefore,
    \begin{align*}
        \lim_{z\to \infty}M'(g(m_q(z)))g'(m_q(z))m_q'(z) &= 
        g'(q)(f(p)-g(q))\left(1-\frac{1}{q\bar{q}}\right) \left(\lim_{z\to\infty}\frac{(1+\frac{z}{q})^{-1}}{g(m_q(z))-g(q)}\right)^2 \\
        &= g'(q)(f(p)-g(q))\left(1-\frac{1}{q\bar{q}}\right) \left(\lim_{z\to\infty}\frac{-(1+\frac{z}{q})^{-2}\frac{1}{q}}{g'(m_q(z))\frac{1-\frac{1}{q\bar{q}}}{(1+\frac{z}{q})^2}}\right)^2 \\
        &= g'(q)(f(p)-g(q))\left(1-\frac{1}{q\bar{q}}\right) \left(\frac{-\frac{1}{q}}{g'(q)(1-\frac{1}{q\bar{q}})}\right)^2 \\
        &= \frac{(f(p)-g(q))\frac{1}{q^2}}{g'(q)(1-\frac{1}{q\bar{q}})} \\
        &= \frac{(f(p)-g(q))}{g'(q)(|q|^2-1)\frac{q}{\bar{q}}}.
    \end{align*}
    Likewise,
    \begin{align*}
        \lim_{z\to 0} M'(f(m_p(z)))f'(m_p(z))m_p'(z) 
        &= \frac{f'(p)(1-|p|^2)}{f(p)-g(q)}.
    \end{align*}
    Thus, 
    $$D(f,g,p,q) = \log\left|\frac{(f(p)-g(q))^2}{f'(p)g'(q)(1-|p|^2)(|q|^2 - 1)}\right| \geq 0.$$
\end{proof}

If the Jordan curve $\gamma$ is asymptotically conformal, we can consider local quasireflection $j(w) = F\left(\frac{1}{\overline{F^{-1}(w)}}\right)$ coming from the local quasiconformal extension $F$ via osculating M\"{o}bius transformations. A direct computation give the following proposition. 
\begin{proposition} \label{prop: dist btw horoball}
    Following the notation in Proposition \ref{prop: general dist btw horoball}, if $\gamma$ is asymptotically conformal,
    \[D(f,g,z,g^{-1}(F(1/\bar{z}))) = \log( \Psi(f(z)))\]
    for all $z\in \Delta\cap A_{r,1/r}$, where $\Psi(w) = \left(\frac{EN(w)}{16}\right)^{1/2}$.
\end{proposition}
As a consequence of Proposition \ref{prop: dist btw horoball} and Theorem \ref{thm: EN converges}, it implies that for an asymptotically conformal curve, the distance between the horoballs based at $w$ and $j(w)$ converges to $0$ with respect to the hyperbolic metric. To conclude the proof of the Theorem \ref{thm: Eps converges}, we need the following lemma.
\begin{lemma} \label{lem: Eps and EN estimate}
    Let $T(w) = \inf \{ t\in \R| \pi(\g_t(\widetilde{\Ep_{\Omega}}(w)))\in \Ep_{\Omega^*}(\Omega^*) \}$. For each domain bounded by an asymptotically conformal curve, there exists an open set $U \subseteq \Omega$ such that $\gamma \subset \partial U$ and $T(w)\leq \log\Psi(w)$ for $w \in U$.
\end{lemma}

\begin{proof}[Proof of Theorem \ref{thm: Eps converges}]
    By Lemma \ref{lem: Grunsky}, $T(w) \geq 0$. Therefore, by Lemma \ref{lem: Eps and EN estimate}, 
    \begin{align*}
        0 \leq d(\Ep_{\Omega}(w),\Ep_{\Omega^*}(\Omega^*)) \leq T(w) \leq \log\Psi(w)
    \end{align*}
    for $w\in U$. The conclusion follows from Theorem \ref{thm: EN converges}.
\end{proof}
%\marginpar{\tiny need to shrink the domain}

For all $w \in \Omega$, let $L_f(w)$ be the set of points on the hyperbolic geodesic in $\HH^3$ that joins $w$ to $j(w)$. Let $p = j(w)$ and $q\in \Ep_{\Omega^*}^{-1}(L_f(w)\cap \Ep_{\Omega^*}(\Omega^*))\subseteq \Omega^*$. Let $\overline{\HH^3} = \HH^3\cup S^2$. 

\begin{lemma} \label{lem: converging of p and q}
    Suppose $w \to \alpha \in \partial \Omega$. Then, the corresponding $p$ and $q$ converge to $\alpha \in \partial(\overline{\HH^3})$ with respect to the Euclidean topology.
\end{lemma}
\begin{proof}
    Since $j$ is a continuous reflection that fixes $\gamma$ pointwise, $p=j(w) \to j(\alpha) = \alpha \in \partial \Omega^*$ as $w \to \alpha \in \partial \Omega$.

    For all sequences $w_n$ that converge to $\alpha \in \partial \Omega$, let $a_k$ be any subsequence of $q_n$. Since $\gamma \cup \Omega^*$ is compact with respect to the spherical topology of $\hat{\C}$, there exists a subsequence of $a_k$ that converges to $\beta \in \gamma \cup \Omega^*$. We can assume $a_k$ converges to $\beta \in \gamma \cup \Omega^*$. Note that each $a_k \in \Omega^*$. 

    Suppose for a contradiction that $\beta\neq \alpha$. If $\beta \in \Omega^*$, then $\{a_k\}\cup \{\beta\}$ is a compact set in $\Omega^*$. Let $\xi:\overline{\HH^3} \to \R_{\geq 0}$ be defined as $(x,y,h) \mapsto h$. Therefore, $\min(\xi(\{\Ep_{\Omega^*}(a_k)\}\cup \{\Ep_{\Omega^*}(\beta)\})) > 0$.
    
    However, since $\Ep_{\Omega^*}(a_k) \in L_f(w_k)$ and $L_f(w_k)$ is the hyperbolic geodesic joining $w_k$ and $j(w_k)$ in which both points converge to the same point $\alpha$, the height of $L_f(w_k)$ converges to $0$. By the continuity of $\Ep_{\Omega^*}$ and its extension to the identity on $\gamma$ (See Corollary 3.13 of \cite{BBVW_universal}), $\xi(\Ep_{\Omega^*}(\beta)) = 0$, a contradiction.
    
    Therefore, $\beta \in \gamma$. By a similar argument as before, $\beta = \alpha$. Since any subsequence of $q_n$ has a further subsequence that converges to $\alpha$, $q_n \to \alpha \in \partial \overline{\HH^3}$.
\end{proof}
Let $w_n \to \alpha \in \partial \Omega$ and let $p_n$ and $q_n$ be the corresponding points in $\Omega^*$ defined above. 

Following the idea of Lemma \ref{lem: bound on length and curv}, choose a similar open set $U_\alpha\subseteq \overline{\Omega^*}$ containing $\alpha$. By Lemma \ref{lem: converging of p and q}, there exists $N>0$ such that for all $n\geq N$, $p_n,q_n\in U_\alpha$. For each $n\geq N$, by post-composing with a M\"{o}bius transformation, we can assume that $w_n = 0$, $p_n = j(w_n) =\infty$ and the horoball $H_{\rho_{\Omega^*}}(p_n)$ is the affine plane $\{\xi = 1\}$. Let $\tau_n:(-\epsilon,l_n+\epsilon) \to \Ep_{\Omega^*}(\Omega^*)$ be a smooth path obtained from Lemma \ref{lem: bound on length and curv} such that $\tau_n(0) = \Ep_{\Omega^*}(p_n)$ and $\tau_n(l) = \Ep_{\Omega^*}(q_n)$. 
\begin{lemma} \label{lem: below Eps}
    For all $0<s<l$, $\tau_n(s)$ lies below $H_{p_n}$. 
\end{lemma}
\begin{proof}
    This follows from Lemma \ref{lem: decreasing Eps} and Lemma \ref{lem: converging of p and q}.
\end{proof}
\begin{proof}[Proof of Lemma \ref{lem: Eps and EN estimate}]
    By Lemma \ref{lem: below Eps}, $c := \xi(\Ep_{\Omega}(w)) \leq e^{h(l)} \leq 1$. Hence, $$T(w) \leq d(\Ep_{\Omega}(w),\Ep_{\Omega^*}(q)) = d(\Ep_{\Omega}(w),\tau(l)) = \log\left(\frac{e^{h(l)}}{c}\right) \leq \log(c^{-1}) = \log \Psi(w).$$   
\end{proof}

\section{Global Bounds on the Function EN(w)} \label{sec: global bound}
In this section, we assume that the uniformization map $f:\Delta\to \Omega$ satisfies $\|Sf\|_\Delta<1/2$; we do not assume that $\gamma$ is asymptotically conformal. Note that under this condition, the quasireflection $j$ defined in the previous section is defined on the whole Riemann sphere. In this section, we will prove Theorem \ref{thm: global bound on EN} and establish a bound on the local section for the Bers embedding for $\Omega$.

\subsection{A Bound on the function EN(w)}
The following lemma establishes the first equality of Theorem \ref{thm: global bound on EN}.
\begin{lemma} \label{lem: partial bar z of j}
    The $\partial_{\bar{w}}$ of the quasireflection $j$ is given by
    \begin{align*}
        \partial_{\bar{w}}j(w) = j_{\bar{w}}(w) = 
        \begin{cases}
            -\frac{\rho_\Omega(w)(w-j(w))^2}{4} &,w\in \Omega, \\
            -\frac{4}{(\overline{w-j(w)})^2\rho_\Omega(j(w)) (1-4\|Sf(\zeta)\|^2_\Delta)} &, w\in \Omega^*
        \end{cases}
    \end{align*}
    where $\zeta = f^{-1}(j(w))$.
\end{lemma}
\begin{proof}
    By the definition of $F$, 
    \begin{align*}
        1-Nf(\frac{1}{\bar{z}})\frac{z-\frac{1}{\bar{z}}}{2} = \frac{f'(\frac{1}{\bar{z}})(z-\frac{1}{\bar{z}})}{F(z) - f(\frac{1}{\bar{z}})}
    \end{align*}
    for all $z\in \Delta^*$. For $w\in \Omega$,
    \begin{align*}
        j_{\bar{w}}(w) = F_z\left(\frac{1}{\overline{z}}\right) \frac{-1}{\bar{z}^2} \overline{(f^{-1})'(w)}
    \end{align*}
    where $z = f^{-1}(w)$. Recall from Section \ref{sec: Hyp and QuasiHyp} that 
    $$F_z(z) = \frac{f'(\frac{1}{\bar{z}})}{(1-Nf(\frac{1}{\bar{z}})\frac{z-\frac{1}{\bar{z}}}{2})^2}$$ 
    for all $z\in \Delta^*$. Therefore,
    \begin{align*}
        j_{\bar{w}}(w) &= \frac{f'(\frac{1}{\bar{z}})}{(1-Nf(\frac{1}{\bar{z}})\frac{z-\frac{1}{\bar{z}}}{2})^2}\frac{-1}{\bar{z}^2} \overline{(f^{-1})'(w)} \\
        &= -\frac{(j(w)-w)^2}{f'(\frac{1}{\bar{z}})(z-\frac{1}{\bar{z}})^2\bar{z}^2\overline{f'(z)}} \\
        &=  -\frac{(j(w)-w)^2\rho_\Omega(w)}{4}.
    \end{align*}

    Likewise, for $w\in \Omega^*$, as in the proof of Lemma \ref{lem: comparison for F}, 
    \begin{align*}
        j_{\bar{z}}(w) = \left(\frac{f'(\frac{1}{\bar{z}})}{\overline{f'(\frac{1}{\bar{z}})}}\right)\left(\frac{-1}{\overline{z}^2}\right)\frac{\left(\overline{1-Nf(\frac{1}{\bar{z}})\frac{(z-\frac{1}{\bar{z}})}{2}}\right)^2}{1-|\mu_F(z)|^2}
    \end{align*}
    where $z = F^{-1}(w)$. Therefore,
    \begin{align*}
        j_{\bar{w}}(w) = -\frac{|f'(\frac{1}{\bar{z}})|^2(1-|\frac{1}{z}|^2)^2}{(\overline{w-j(w)})^2(1-|\mu_F(z)|^2)} = -\frac{4}{(\overline{w-j(w)})^2\rho_{\Omega}(j(w))(1-|\mu_F(z)|^2)}.
    \end{align*}
    Note that $|\mu_F(z)|^2 = 4\|Sf(\frac{1}{\bar{z}})\|_\Delta^2$. 
\end{proof}
To show the second inequality for Theorem \ref{thm: global bound on EN}, we consider the conformal metric $\rho(w)|dw|^2$ on $\Omega^*$ with $\rho(w) = \frac{1}{|w-j(w)|^4\rho_\Omega(j(w))} = \frac{\rho_\Omega(j(w))}{16|j_{\bar{w}}(j(w))|^2}$. The following lemma summarizes the properties of this metric:
\begin{lemma}\label{lem: properties of h}
    Assume that $\|Sf\|_\Delta<1/2$, where $f$ is a uniformization map for $\Omega$. Let $\rho(w)|dw^2| = \frac{\rho_\Omega(j(w))}{|j_{\bar{w}}(j(w))|^2} |dw|^2$ be the metric on $\Omega^*$ where $j$ is the quasireflection given in this section. Then, $\rho$ is a complete metric with curvature
    \begin{align*}
        K_\rho(w) = \frac{-16}{1-4\|Sf(z)\|_\Delta^2}
    \end{align*}
    where $z = f^{-1}(j(w))$.
\end{lemma}
\begin{proof}
    By pulling back the metric $\rho$ by $j$, we obtain
\begin{align*}
    (1-2\|Sf\|_\Delta)^2\rho_\Omega(w)|dw|^2 &\leq \frac{\rho_\Omega(w)}{|j_{\bar{w}}(w)|^2}|j_w dw + j_{\bar{z}} d{\bar{w}}|^2 = j^*(\rho)  \\
    &= \rho_\Omega(w)\left|\frac{j_w}{j_{\bar{w}}} dw + d{\bar{w}}\right|^2 \leq (1+2\|Sf\|_\Delta)^2\rho_\Omega(w) |dw|^2
\end{align*}
for $w\in \Omega$. Hence, $\rho$ is a complete metric. The curvature of $\rho$ is given by 
\begin{align*}
    K_\rho = -\frac{2 \partial_w\partial_{\bar{w}}\log \rho}{\rho} 
    = \frac{2\partial_w\partial_{\bar{w}}\log(|w-j(w)|^4\rho_\Omega(j(w)))}{\rho}.
\end{align*}
By Lemma \ref{lem: naturality of osc} and the computation in the Appendix, 
\begin{align*}
    j(w) = w + \frac{2}{\overline{(\log\rho_\Omega(w))_{\bar{w}}}} \\
    \iff (\log \rho_\Omega(w))_{\bar{w}} - \frac{2}{\overline{w-j(w)}} = 0 \\
    \iff (\log \rho_\Omega(w))_{w} - \frac{2}{w-j(w)} = 0
\end{align*}
for $w\in \Omega$ as $\rho_\Omega$ is real valued. Therefore,
\begin{align*}
    \partial_{\bar{w}}\log(|w-a|^4\rho_\Omega(a)) &= 
    \partial_{\bar{w}} (2\log (w-a) + 2\log(\overline{w-a}) + \log(\rho_\Omega(a))\\ 
    &= \frac{-2a_{\bar{w}}}{w-a} + \frac{2}{\overline{w-a}} -\frac{2\bar{a}_{\bar{w}}}{\overline{w-a}} + (\log(\rho_\Omega(a)))_a a_{\bar{w}} + (\log(\rho_\Omega(a)))_{\bar{a}} \bar{a}_{\bar{w}} \\
    &= \frac{2}{\overline{w-a}}
\end{align*}
where $a=j(w)$. By Lemma \ref{lem: partial bar z of j}, its $\partial_w$ derivative gives
\begin{align*}
    \partial_w \left(\frac{2}{\overline{w-j(w)}}\right) &= -2\frac{\overline{j(w)_{\bar{w}}}}{(\overline{w-j(w)})^2} \\ &= \frac{8}{|w-j(w)|^4\rho_\Omega(j(w)) (1-4\|Sf(z)\|^2_\Delta)} \\
    &= \frac{8\rho(w)}{1-4\|Sf(z)\|^2_\Delta}
\end{align*}
where $z = f^{-1}(j(w))$. 
\end{proof}
The last ingredient comes from the following theorem of Yau:
\begin{theorem}[See Theorem 2' of \cite{yau1978general}] \label{thm: yau schwarz lemma}
    Let $M$ be a complete Riemann surface with curvature bounded from below by a constant $K_1$. Let $N$ be another Hermitian manifold with holomorphic sectional curvature bounded from above by a negative constant $K_2$. Then for any nonconstant holomorphic function $f:M \to N$,
    \begin{align*}
        f^*dS_N^2 \leq \frac{K_1}{K_2}dS_M^2.
    \end{align*}
\end{theorem}
\begin{proof}[Proof of Theorem \ref{thm: global bound on EN}]
    Recall our goal of showing that $$EN(w) = 4|w-j(w)|^2\rho_{\Omega^*}(j(w))||j_{\bar{z}}(w)| \leq \frac{16}{1-4\|Sf\|_\Delta^2}$$
    for $f:\Delta \to \Omega$ a uniformization map for $\Omega$ and $\|Sf\|_\Delta < \frac{1}{2}$. The first equality is shown in Lemma \ref{lem: partial bar z of j}. For the second inequality, by Theorem \ref{thm: yau schwarz lemma}, we apply it to the identity map $(\Omega^*,\rho)\to (\Omega^*,\rho_{\Omega^*})$, we obtain $\rho_{\Omega^*}(w) \leq \frac{16\rho(w)}{1-4\|Sf\|_\Delta^2}$. Hence, the theorem follows from $EN(j(w)) = \frac{\rho_{\Omega^*}(w)}{\rho(w)}$ for all $w\in\Omega^*$. 
\end{proof}

\subsection{Local Section for the Bers Embedding}
In this subsection, we give a brief overview of Bers Embedding and establish bounds related to this embedding. Let $\Omega$ be a Jordan domain bounded by a quasicircle and $L_1^\infty(\Omega)$ be the space of measurable functions on $\Omega$ with $L^\infty$ norm less than $1$. The universal Teichm\"{u}ller space over $\Omega$ is given by 
\begin{align*}
    T(\Omega) = L_1^\infty(\Omega)/\sim
\end{align*}
where $\mu \sim \nu$ if and only if $F_\mu|_{\Omega^*} = F_\nu|_{\Omega^*}$ with $F_\mu$ being the solution to the equation $\partial_{\bar{z}} F_\mu = \mu^\vee \partial_z F_\mu$, $\mu^\vee$ is the extension of $\mu$ by $0$ to $\hat{\C}$ subject to a normalization (See \cite{AhlforsBersMapping} for detail). The Bers embedding is an embedding of $T(\Omega)$ into the space of quadratic differentials $Q(\Omega^*)$ via the Schwarzian derivative. More explicitly, the embedding $\beta:T(\Omega) \to Q(\Omega^*)$ is defined by
\begin{align*}
    \beta([\mu]) = S(F_\mu|_{\Omega^*}).
\end{align*}

Let $Q_C(\Omega) = \{q\in Q(\Omega)\mid\|q\|_\Omega < C\}$ where $\|q\|_\Omega = \sup_{w\in \Omega} \frac{|q(w)|}{\rho_{\Omega}(w)}$. Earle and Nag \cite{Earle:qcreflection} showed that for appropriate constant $C$ depending on the quasicircle, for $j$ the quasireflection that is equivariant under the group of M\"{o}bius transformation, the Ahlfors map $\alpha: Q_C(\Omega) \to L_\infty(\Omega^*)$ given by 
\begin{align*}
    \alpha(\phi)(w) = \frac{\phi(j(w)) (w-j(w))^2j_{\bar{z}}(w)}{2+\phi(j(w))(w-j(w))^2 j_z(w)}
\end{align*} 
for $w\in \Omega^*$ is well-defined. The following summarizes the origin of this map:
For $\phi \in Q_C(\Omega)$, we can find a locally univalent map $u:\Omega \to \C$ such that $Sh = \phi$. Consider the extension of $h$ by
\begin{align*}
    H(z) = \begin{cases}
        h(z) &, z\in \Omega, \\
        Mh(j(z),z) &, z\in \Omega^*.
    \end{cases}
\end{align*}
Then $\alpha(\phi)(z) = \mu_H(z)$ for $z\in \Omega^*$. It is worth mentioning (See Theorem 4 of \cite{Earle:qcreflection}) that when $\|\alpha(\phi)\|_\infty < 1$, $H$ is a quasiconformal homeomorphism of $\hat{\C}$ with $SH|_\Omega = \phi$, therefore giving a section from $Q_C(\Omega)$ to $L_1^\infty(\Omega^*)$. 

The following corollary gives a bound on $C$ in terms of the Schwarzian norm of the uniformization map $f:\Delta\to \Omega$.
\begin{corollary}\label{cor: ahlfors section for omega}
    Let $f:\Delta \to \Omega$ be a uniformization map of $\Omega$ with $\|Sf\|_\Delta<1/2$. If $C = \frac{1}{2} - \|Sf\|_\Delta$, then $\alpha_C = \alpha|_{Q_C(\Omega)}:Q_C(\Omega)\to L_1^\infty(\Omega^*)$ is a well-defined section.
\end{corollary}
\begin{proof}
    Let $C=\frac{1}{2} - \|Sf\|_\Delta$. By Lemma \ref{lem: partial bar z of j}, for any $\phi \in Q_C(\Omega)$, \begin{align*}
        |\widetilde{\alpha}(\phi)(w)| &= \left|\frac{\phi(j(w)) (w-j(w))^2j_{\bar{z}}(w)}{2+\phi(j(w))(w-j(w))^2 j_z(w)} \right| \\
        &\leq \frac{\|\phi\|_\Omega \frac{4}{1-4\|Sf\|_\Delta^2}}{\left|2-\|\phi\|_{\Omega}\frac{4}{1-4\|Sf\|_\Delta^2} \|\frac{j_z}{j_{\bar{z}}}\|\right|} \\
        &= \frac{2\|\phi\|_\Omega}{|1-4\|Sf\|_\Delta^2-4\|\phi\|_{\Omega}\|Sf\|_\Delta|} \\
        &< \frac{1-2\|Sf\|_\Delta}{1-4\|Sf\|_\Delta^2 - 2\|Sf\|_\Delta + 4\|Sf\|_\Delta^2} = 1
    \end{align*}
\end{proof}
Corollary \ref{cor: ahlfors section for omega} has the following interpretation. For any $\phi \in Q_C(\Omega)$ with $C = \frac{1}{2} - \|Sf\|_\Delta$, let $u:\Omega\to \C$ be a solution to $Su = \phi$. Then 
$$\|S(u\circ f)\|_\Delta \leq \|Su\|_\Omega + \|Sf\|_\Delta < \frac{1}{2}.$$ By the Ahlfors--Weill Theorem, we can extend $h := u\circ f$ by osculating M\"{o}bius transformation to a quasiconformal homeomorphism $H$. Similarly, let $F$ be the Ahlfors--Weill extension of $f$. Therefore, $H\circ F^{-1}$ is a quasiconformal extension of $u$. In fact, for all $w\in \Omega^*$,
\begin{align*}
    H(F^{-1}(w)) &= Mh\left(\frac{1}{\overline{F^{-1}(w)}},F^{-1}(w)\right) \\
    &= Mu\left(f\left(\frac{1}{\overline{F^{-1}(w)}}\right),Mf\left(\frac{1}{\overline{F^{-1}(w)}},F^{-1}(w)\right)\right) \\
    &= Mu(j(w),w).
\end{align*}
Therefore, the Beltrami differential of $H\circ F^{-1}|_{\Omega^*}$ is equal to $\alpha(\phi)$ and $\|\alpha(\phi)\|_\infty<1$ since $H\circ F^{-1}$ is a quasiconformal homeomorphism.

Likewise, we can consider $L_1^\infty(\Omega)$ and the corresponding local section $\widetilde{\alpha}_C:Q_C(\Omega^*)\to L_1^\infty(\Omega)$. By an argument similar to that of Corollary \ref{cor: ahlfors section for omega}, we obtain the following corollary:
\begin{corollary} 
    Let $f:\Delta \to \Omega$ be a uniformization map of $\Omega$ with $\|Sf\|_\Delta<1/2$. If $C = \frac{1}{2}-\|Sf\|_\Delta$, then $\widetilde{\alpha}_C:Q_C(\Omega^*)\to L_1^\infty(\Omega)$ is a well-defined section.
\end{corollary}

\section{Appendix}

\subsection{Metric corresponding to extension $P$} In this section, we will give the metric that is naturally related to the extension $P$. 

Recall that, given a univalent map $f:\Delta \to \C$, the map $P:\hat{\C} \to \hat{\C}$ is given by
\begin{align*}
    P(z) = \begin{cases}
        f(z) &,\text{if } z\in \overline\Delta, \\
        f(\frac{1}{\bar{z}}) + f'(\frac{1}{\bar{z}})(z-\frac{1}{\bar{z}}) &,\text{if } z\in \Delta^*.
        \end{cases}
\end{align*}
Following the notation in Proposition \ref{prop: properties of general Ep}, we want to find a smooth conformal metric $\rho = e^{\varphi}|dw|^2$ on $\Omega = f(\Delta)$ such that 
$$\lim_{t\to -\infty} \g_t(\widetilde{\Ep_\rho}(w))= P\left(\frac{1}{\overline{P^{-1}(w)}}\right)$$
for all $w\in \Omega$. We proceed by examining the explicit formula for $\Ep_\rho := \pi(\widetilde{\Ep_\rho}):\Omega \to \HH^3$ where $\pi:\mathsf{U}\HH^3 \to \HH^3$ is the projection onto its base point. From Section 3.3 of \cite{BBVW_universal}, $\Ep_\rho(w) = (W,\xi) \in \C \times \R_{>0}$ where 
\begin{align*}
    \xi = \frac{2e^{-\varphi/2}}{1+|\varphi_{\bar{w}}|^2e^{-\varphi}}, \qquad W = w+\xi\varphi_{\bar{w}}e^{-\varphi/2}. 
\end{align*}
Let $\xi_t = \frac{2e^{-\varphi/2-t}}{1+|\varphi_{\bar{w}}|^2e^{-\varphi-2t}}$. By Proposition \ref{prop: properties of general Ep}, 
\begin{align*}
    \lim_{t\to -\infty} \g_t(\widetilde{\Ep_\rho}(w))= \lim_{t\to -\infty}(w+\xi_t\varphi_{\bar{w}}e^{-\varphi/2-t}, \xi_t) = \left(w+\frac{2}{\overline{\varphi_{\bar{w}}}},0\right). 
\end{align*}
Therefore, as $\varphi$ is real-valued, 
\begin{align*}
    \varphi_{w}(w) &= \overline{\varphi_{\bar{w}}}(w) \\
    &= \frac{2}{P\left(\frac{1}{\overline{P^{-1}}(w)}\right)-w} \\
    &= \frac{2}{f'(f^{-1}(w))\left(\frac{1}{\overline{f^{-1}}(w)}-f^{-1}(w)\right)} \\
    &= \frac{2(f^{-1})'(w)\overline{f^{-1}}(w)}{1-|f^{-1}(w)|^2}.
\end{align*}
As $f$ is holomorphic, $\varphi(w) = -2\log(1-|f^{-1}(w))|^2)+Q(w)$ for some anti-holomorphic function $Q$. Since $\varphi$ is real-valued, $Q(w) = \bar{Q}(w)$ is a real-valued holomorphic function which, by open mapping theorem, implies that $Q(w)$ is a constant function in $\R$. Therefore, a metric corresponding to $P$ is given by 
\begin{align*}
    \rho = \frac{4}{(1-|f^{-1}(w)|^2)^2} |dw|^2.
\end{align*}

\printbibliography

\end{document}